\documentclass[letterpaper,11pt]{amsart}
\usepackage[margin=1in]{geometry}
\usepackage{amsfonts}
\usepackage{comment}
\usepackage{graphicx,caption,float}
\usepackage{url}
\title[Algebraic Nielsen--Thurston Classification]{An effective algebraic detection of the Nielsen--Thurston classification of mapping classes}

\newtheorem{thm}{Theorem}
\newtheorem{lemma}[thm]{Lemma}

\newtheorem{cor}[thm]{Corollary}
\newtheorem{prop}[thm]{Proposition}

\numberwithin{thm}{section}

\newcommand{\al}{\alpha}

\newcommand{\bZ}{\mathbb{Z}}

\newcommand{\bC}{\mathbb{C}}

\DeclareMathOperator{\Aut}{Aut}
\DeclareMathOperator{\Out}{Out}

\DeclareMathOperator{\Mod}{Mod}

\newcommand{\bH}{\mathbb{H}}

\newcommand{\yt}{\widetilde}

\newcommand{\mW}{\mathcal{W}}

\newcommand{\grp}[1]{\langle{#1}\rangle}

\newcommand{\del}{\partial}

\author[T. Koberda]{Thomas Koberda}
\address{Department of Mathematics\\ Yale University\\ P.O. Box 208283\\ New Haven, CT 06520-8283}
\email{ thomas.koberda@gmail.com}
\author[J. Mangahas]{Johanna Mangahas}
\address{Department of Mathematics\\ Brown University\\ 151 Thayer St.\\ Providence, RI 02912 }
\email{ mangahas@math.brown.edu}

\begin{document}
\bibliographystyle{amsalpha}

\maketitle
\begin{center}
\today
\end{center}
\begin{abstract}
In this article, we propose two algorithms for determining the Nielsen-Thurston classification of a mapping class $\psi$ on a surface $S$.  We start with a finite generating set $X$ for the mapping class group and a word $\psi$ in $\langle X \rangle$.  We show that if $\psi$ represents a reducible mapping class in $\Mod(S)$ then $\psi$ admits a canonical reduction system whose total length is exponential in the word length of $\psi$.  We use this fact to find the canonical reduction system of $\psi$.  We also prove an effective conjugacy separability result for $\pi_1(S)$ which allows us to lift the action of $\psi$ to a finite cover $\yt{S}$ of $S$ whose degree depends computably on the word length of $\psi$, and to use the homology action of $\psi$ on $H_1(\yt{S},\mathbb{C})$ to determine the Nielsen-Thurston classification of $\psi$.
\end{abstract}


\thispagestyle{empty}
\section{Introduction}
In this paper we develop two algorithms for determining the Nielsen--Thurston classification of an element of the mapping class group $\Mod(S)$.  Both pivot on our ability to compute an upper bound on the length of the \emph{canonical reduction system} $\sigma(f)$ of a given mapping class $f$, which we state as our first theorem.  We define $\sigma(f)$ in Section \ref{s:back}, noting here that it consists of nonperipheral curves fixed by powers of $f$, and is empty exactly when $f$ has finite order or is pseudo-Anosov.  We clarify what constitutes a \emph{reasonable definition of curve length} just after Proposition \ref{p:dehnhumphlick} below. 

\begin{thm}\label{t:estimate} 
Let $X \subset \Mod(S)$ be a finite generating set.  There exists computable $C = C(X,S,\ell)$ such that, for any $f \in \mathrm{Mod}(S)$ and $\alpha \in \sigma(f)$, we have $\ell(\alpha) \leq C^{\ell_X(f)}$, where $\ell_X$ is word length with respect to $X$ and $\ell(\alpha)$ is any reasonable definition of curve length.
\end{thm}

The impetus for Theorem \ref{t:estimate} came from a challenge to make effective the result in \cite{KobGeom} that, given $f \in \Mod(S)$, one may find a finite cover $\yt{S}$ such that the action on $H_1(\yt{S},\bC)$ induced by lifts of $f$ to $\yt{S}$ determines the Nielsen--Thurston classification of $f$.  Call such $\yt{S}$ a \emph{classifying cover} for $f$.  Our first algorithm can be stated as a theorem in the following way:

\begin{thm}[Cover Algorithm]\label{t:covers} 
Suppose $f \in \Mod(S)$ is given, and $\yt{S}$ is its smallest classifying cover.  The degree of $\yt{S}$ over $S$ is bounded above by a computable function of the word length of $f$.
\end{thm}

Our second algorithm uses Theorem \ref{t:estimate} to directly determine $\sigma(f)$:

\begin{thm}[List Algorithm]\label{t:list} 
There exists an algorithm, exponential-time when $S$ is punctured, which takes as input a mapping class $f\in\Mod(S)$, presented as a word in $X$, and outputs $\sigma(f)$.  When $\sigma(f)$ is empty the algorithm determines whether $f$ is finite-order or pseudo-Anosov.
\end{thm}

Existence of $C$ in Theorem \ref{t:estimate} follows from Theorem E in Tao's solution to the conjugacy problem for $\mathrm{Mod}(S)$, which relies on heavy machinery related to the curve complex \cite{Tao}.  Now that key curve complex constants have been determined (the smallest in \cite{HPW, Webb}, but see also \cite{Aougab, Bowditch, CRS}), it may be of interest to try to compute $C$ by this approach.  In contrast, we use relatively elementary methods to show:

\begin{prop}\label{p:dehnhumphlick} 
Suppose $S$ has $g \geq 0$ genus and $n \geq 0$ punctures.  There exists a set $X$ generating $\Mod(S)$ and a filling set of curves $\{\gamma_j\}$ so that, if curve length is defined by \[\ell(\alpha) = \sum_j i(\alpha,\gamma_j),\] where $i(\cdot,\cdot)$ denotes geometric intersection number, then for any $\alpha \in \sigma(f)$, we have
\begin{equation*}\label{e:DHLub}
\ell(\alpha) \leq A \cdot B^{217\chi(S)^2\ell_X(f)}
\end{equation*}
with $A = 2g+1$ and $B = 2$ if $n \leq 1$, otherwise with $A = 2g + 2n - 1$ and $B=3$.
\end{prop}

By \emph{reasonable definition of curve length} we mean a length estimate which is quasi-isometric to $\ell(\alpha)$ as defined above, which we note includes length according to any hyperbolic metric on $S$ or, fixing any base point $p$ and choice of word length on $\pi_1(S,p)$, minimal length of elements in the conjugacy class in $\pi_1(S,p)$ represented by $\alpha$ (for non--peripheral curves only).  Theorem \ref{t:estimate} follows easily from Proposition \ref{p:dehnhumphlick}, which is proved in Section \ref{s:computation} (see Theorem \ref{t:computation}).

When $S$ has boundary, the List Algorithm indicated by Theorem \ref{t:list} can be implemented in exponential time, as described in Section \ref{s:complexity}, using compressed representations of curves on $S$ such as intersection with a fixed triangulation, or Dehn--Thurston coordinates.  The main points of interest in the results here are: the conceptual simplicity of the List Algorithm, the connection to homology classification afforded by the Cover Algorithm, and the elementary nature of our proof of Theorem \ref{t:estimate}, which relies mainly on the pigeon-hole principal and easy observations about intersection number, given Ivanov's description of canonical reduction systems of mapping classes.

Both our algorithms are independent of the famous algorithm of Bestvina and Handel, which uses train tracks and Perron--Frobenius theory to prove and implement Nielsen--Thurston classification for $\Mod(S)$ \cite{BH2}.  This algorithm outputs a \emph{train track} for the input element, which is more information than $\sigma(f)$.  While this algorithm runs quickly in practice, its precise theoretical complexity is unclear.  Recently, Mark Bell has developed an exponential-time classification algorithm based on triangulations of punctured surfaces \cite{Bell}, and has implemented it in his program \textsc{Flipper} \cite{Bell2}.  Hamidi-Tehrani and Chen obtained an exponential-time classification of $\Mod(S)$ using its piecewise-linear action on measured train tracks \cite{HTC}.  Polynomial-time classification is available for braid groups \cite{Calvez}, which are closely related to mapping class groups of punctured disks and admit the same Nielsen--Thurston trichotomy.  Another closely related group, $\Out(F_n)$, admits a similar, albeit more complicated classification (including an analogous train track construction in \cite{BH1}); an algorithmic approach modeled on Theorems \ref{t:estimate} and \ref{t:list} appears in \cite{CMP}.


\subsection{Acknowledgements.}  The authors would like to thank J. Malestein for asking whether there exists an effective method of determining the Nielsen--Thurston classification of a mapping class via examination of the homology of finite covers.  They would also like to thank C. Atkinson, D. Canary, M. Clay, W. Dicks, D. Futer, D. Margalit, H. Namazi, A. Pettet, I. Rivin, S. Schleimer, J. Souto, and D. Thurston for productive conversations.  The first author was supported by an NSF Graduate Student Research Fellowship for part of the time this research was carried out, and by NSF Grant DMS-1203964.  The second author acknowledges support from NSF DMS-1204592 and AMS-Simons.  The authors are grateful to the Banff International Research Station, CRM Barcelona, and the Polish Academy of Sciences for hospitality while this work was completed.  An abstract of this project will appear in the series, Research Perspectives CRM Barcelona.  Finally, the authors thank an anonymous referee for numerous helpful comments.

\section{Background}\label{s:back}
We consider orientable surfaces $S$ with $g$ genus, $n$ punctures, and negative, finite Euler characteristic $\chi(S) = 2 - 2g - n$.  By a \emph{curve} on $S$, we mean an isotopy class of homotopically nontrivial, non-boundary-parallel simple loops.  A \emph{filling set} of curves $\{\gamma_i\}$ is such that any curve $\alpha$ has nonzero geometric intersection $i(\alpha,\gamma_j)$ for some $j$.  As in Proposition \ref{p:dehnhumphlick}, we measure the length of a curve $\alpha$ is by its intersection with a fixed filling set of curves: $ \ell(\alpha) = \sum_j \ell_i(\alpha,\gamma_j)$.  Note that, if $\ell_h$ is length with respect to some fixed hyperbolic metric $h$, then
\[(C_h/2)\cdot\ell(\alpha) \leq \ell_h(\alpha) \leq D_h\cdot\ell(\alpha),\] where $D_h$ is the maximum diameter of the components of $S-\bigcup_i \gamma_i$, and $C_h$ is the minimal thickness of collar neighborhoods around the $\{\gamma_i\}$ chosen thinly enough so that no more than two collar neighborhoods of distinct elements of $\{\gamma_i\}$ cover any one point on $S$. 

\subsection{The Nielsen--Thurston classification of mapping classes}\label{s:NTback}
Let $f\in\Mod(S)$.  Since $f$ is a homeomorphism of $S$ up to isotopy, we have that $f$ acts on curves on $S$.  We have the following trichotomy:
\begin{enumerate}
\item
The mapping class $f$ has  \emph{finite order} as an element of $\Mod(S)$.  Thus, some power of $f$ is isotopic to the identity.
\item
The mapping class $f$ has infinite order in $\Mod(S)$ and there is a curve on $S$ with a finite orbit under the action of $f$.  Such a mapping class is called  \emph{reducible}.
\item
For any curve $\alpha$, the $f$-orbit of $\alpha$ is infinite.  Such a mapping class is called  \emph{pseudo-Anosov}.
\end{enumerate}

The trichotomy above is called the  \emph{Nielsen--Thurston classification}.  The following are standard results about mapping classes, and proofs can be found in \cite{BLM}, \cite{FM} and \cite{FLP}, for instance.

\begin{thm}[Nielsen]
Let $f$ be a finite order mapping class.  Then there exists a hyperbolic metric on $S$ and a representative of $f$ which is an isometry of the corresponding Riemann surface.  In particular, the representative has finite order as a homeomorphism of $S$.
\end{thm}

\begin{thm}
Let $f$ be a reducible mapping class.  Then there exists a multicurve (i.e. a finite set of distinct, disjoint curves) whose isotopy class is setwise stabilized by $f$.
\end{thm}

Such a multicurve is called a \emph{reduction system}.  We will be interested in a mapping class's \emph{canonical reduction system} $\sigma(f)$, which may be defined as the set of curves common to all maximal reduction systems \cite{BLM}.   It is detailed in \cite{Iv2} that, for some $k$ depending only on the surface, the power $f^k$ sends each component of $S\setminus\sigma(f)$ to itself, inducing on each either a trivial or pseudo-Anosov mapping class.  Furthermore, any curve in $\sigma(f)$ that bounds only components on which $f^k$ acts trivially, must be the core curve of an annulus where $f^k$ acts locally as a Dehn twist.

\subsection{Artin--Dehn--Humphries--Lickorish generators}\label{s:DHL}
For a surface $S$ with genus $g \geq 0$ and punctures $n \geq  0$, we can generate $\Mod(S)$ with $2g + n$ Dehn twists ($2g+1$ when $n=0$) and, when $n>1$, an additional $n-1$ \emph{half-twists}, along the curves shown in Figure \ref{fig1}.  Precise definitions and proof can be found in \cite{FM}. We remark that a smaller, similar generating set appeared first in \cite{LP} (Corollary 2.11 there), and in \cite{BD} one can find an alternate, algebraic proof that these generate, along with a detailed history of the project of generating the mapping class group.

Here we note our labeling convention.  For $1 \leq i \leq \max\{2g + n,2g+1\}$, we let $\gamma_i$ be the base curves of the Dehn twists, denoted $T_{\gamma_i}$.  When $n>1$, we have $2g + n < i \leq 2g + 2n - 1$ for which $\gamma_i$ denotes curves corresponding to half-twists in the following way: one component of $S$ cut along $\gamma_i$ is a disk with two punctures, and the half-twist $H_{\gamma_i}$ permutes these two punctures and is supported on this disk.  In this latter case, we call $\gamma_i$ the base curve for the half-twist $H_{\gamma_i}$.
\begin{figure}[H]
\begin{center}
\includegraphics[height=2in]{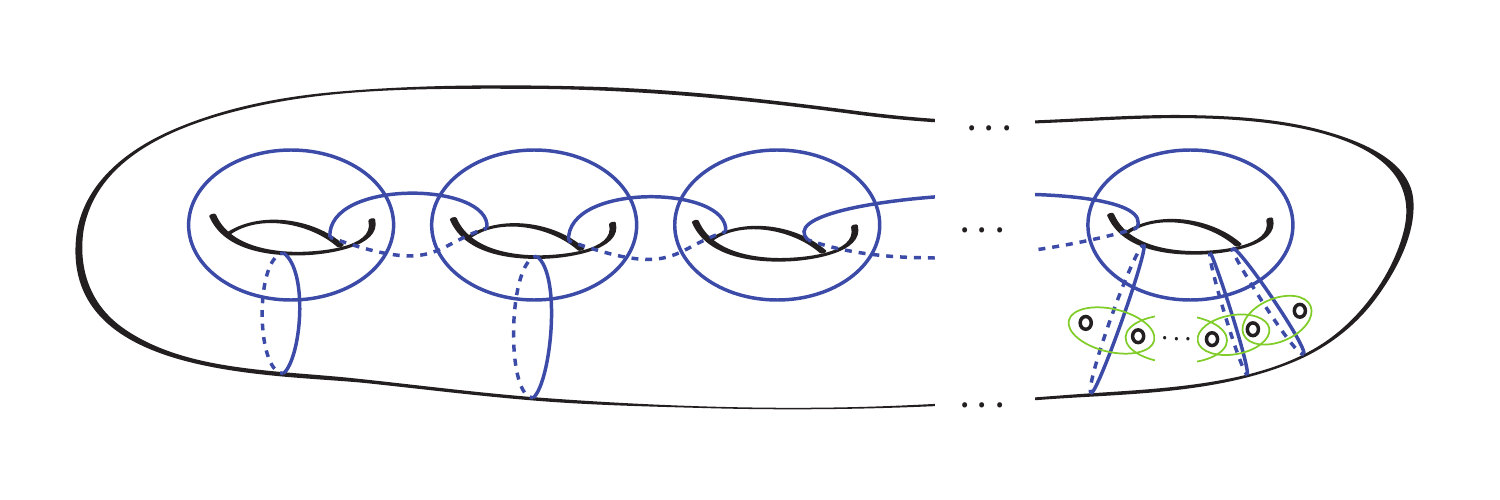}
\caption{The mapping class group is generated by Dehn twists and half-twists about the curves shown.  The curves bounding punctures are the base curves for half-twists.  Graphic based on Figures 4.10 and 9.5 in \cite{FM}.}
\label{fig1}
\end{center}
\end{figure}

\section{The main technical tool}\label{s:computation}

Here we show how to compute $C$ in Theorem \ref{t:estimate}.  Consider the special case that the generating set $X$ consists of Dehn twists or half-twists about separating curves, an example of which is described in Section \ref{s:DHL}.  Let $\Gamma$ be the set of base curves for the Dehn twists and half-twists in $X$.  Define $M = M(X)$ by \[M(X) = \max\{i(\gamma_i,\gamma_j): \gamma_i, \gamma_j \in \Gamma\}.\]  The generators described in Section \ref{s:DHL} give optimal $M=1$ when $S$ has no more than one puncture, and $M=2$ otherwise.  Let $\ell(f) = \ell_X(f)$ denote word length in the mapping class group with respect to $X$.  In this section we prove

\begin{thm}\label{t:computation} For any component $\alpha$ of $\sigma(f)$ and any $\gamma \in \{\gamma_i\}$, \[i(\alpha,\gamma)\leq (M+1)^{217\chi(S)^2\ell(f)}\] \end{thm}

\subsection{Proof of Proposition \ref{p:dehnhumphlick} using Theorem \ref{t:computation}}  This follows immediately by using the generators described in Section \ref{s:DHL} for $X$.

\subsection{Proof of Theorem \ref{t:estimate} using Proposition \ref{p:dehnhumphlick}}  Different generating sets for $\Mod(S)$ change word lengths by a quasi-isometry, and reasonable definitions of curve length also differ from $\ell$ by a quasi-isometry.  The proposition and the quasi-isometry constants allow one to compute $C$.

\subsection{Proof of Theorem \ref{t:computation}}

\begin{lemma}\label{l:upperboundint}For any $\gamma_i$ and $\gamma_j$, we have $i(f(\gamma_i),\gamma_j) \leq M(M+1)^{\ell(f)}.$\end{lemma}

\noindent \emph{Proof.} This follows from induction, using the well-known fact that \[i(T_\gamma(\alpha),\beta)\leq i(\alpha,\gamma)i(\beta,\gamma) + i(\alpha,\beta)\] and the easily verified fact that \[i(H_\gamma(\alpha),\beta)\leq \frac{i(\alpha,\gamma)}{2}\cdot \frac{i(\beta,\gamma)}{2}\cdot 2 + i(\alpha,\beta)\] where $H_\gamma$ is the half-twist with base curve $\gamma$.\qed \bigskip

Now suppose $\alpha$ is a component of $\sigma(f)$.  There are two possibilities: either $\alpha$ is in the boundary of a component $S'$ of $S\setminus\sigma(f)$ on which some power of $f$ acts as a pseudo-Anosov, or $\alpha$ is the core curve around which some power of $f$ acts as a twist.  We deal with each case separately.

\subsection*{Case with pseudo-Anosov component.} Suppose $\alpha$ is a component of $\del S'$.  The maximum number of distinct isotopy classes of disjoint arcs on $S'$ is $3|\chi(S')|$, the number of edges in an ideal triangulation of $S'$ with its boundary shrunk to punctures.  Because every point of intersection of $\gamma$ with $\alpha$ has a component of $\gamma \cap S'$ on at least one side, there are at least $i(\alpha,\gamma)/2$ components of $\gamma \cap S'$.  Thus at least $i(\alpha,\gamma)/(6|\chi(S')|)$ of these run parallel.  Suppose we can find a power $P$ such that, for any essential arc $a$ through $S'$, we have $f^p(a)$ intersects $a$ for some $p \leq P$ (in particular, $f^p$ also fixes $S'$).  If $n$ arcs of $\gamma$ run parallel to $a$, these guarantee $n^2$ points of intersection between $f^p(\gamma)$ and $\gamma$.  Putting these estimates together, we have
\[(M+1)^{P\ell(f)+1} \geq M(M+1)^{\ell(f^p)}\geq i(f^p(\gamma),\gamma)\geq [i(\alpha,\gamma)/(6|\chi(S')|)]^2.\]

For this case, it remaines to compute $P$.  Let $a$ be the arc through $S'$ of interest.  Observe that there are at most $|\chi(S)/\chi(S')|$ copies of $S'$ possibly permuted by $f$, so for some $p_1 \leq |\chi(S)/\chi(S')|$, $f^{p_1}$ preserves $S'$.  Only $3|\chi(S')|$ distinct arc isotopy classes can be simultaneously disjoint, so for some $p_2 \leq 3|\chi(S')|$, $f^{p_1p_2}(a)$ intersects $a$.  Thus we can take $P = 3|\chi(S)|$ in the inequality above, and deduce
\begin{equation}\label{eq:pA}
6|\chi(S)|(M+1)^{(3|\chi(S)|\ell(f)+1)/2}\geq 6|\chi(S')|(M+1)^{(3|\chi(S)|\ell(f)+1)/2}\geq i(\alpha,\gamma).
\end{equation}

\subsection*{Dehn twist case}For the second case, we assume that $\alpha$ is not the boundary of a pseudo-Anosov component for $f$, so some power of $f$ acts as Dehn twisting around $\alpha$.  That is to say, let $\sigma' = \sigma(f)\setminus\{\alpha\}$, and let $S'$ be the component of $S\setminus\sigma'$ containing $\alpha$.  Some power of $f$ induces in $\Mod(S')$ a power of a Dehn twist about $\alpha$.  Suppose $p$ is this power, and consider $f^{3p}$.  We may decompose this as $f^{3p} = g^kh = hg^k$, where $g$ is a single Dehn twist about $\alpha$, its power $k$ is a nonzero multiple of three, and finally $h$ fixes each component of $\del S'$ while inducing the identity map in $\Mod(S')$.  

We claim that $i(f^{3p}\gamma,\gamma) \geq i(\alpha,\gamma)^2/2$.  This can be seen by lifting to the universal cover.  Consider the pre-image $P$ of geodesic representatives of $\del S' \cup \alpha$ in the universal cover, and choose a particular lift $\yt{\alpha}$ of $\alpha$.  Letting $C_-$ and $C_+$ be the components of $\mathbb{H}^2\setminus P$ on each side of $\yt{\alpha}$, we may choose a lift $\yt{g}$ of $g$ which fixes $C_-$ and acts on $C_+$ as the deck transformation $D_\alpha$ corresponding to $\alpha$.  Let $\yt{h}$ be a lift of $h$ that also fixes $C_-$, so that we may lift $f^{3p}$ to $\yt{f^{3p}} = \yt{h}\yt{g}^k$.

We choose $n=i(\alpha,\gamma)$ consecutive lifts $\{\yt{\gamma_i}\}_{i=1}^n$ of $\gamma$ intersecting $\yt{\alpha}$ on the restriction of a fundamental domain.  By examining the endpoints at infinity, one can see that the geodesic representative of $\yt{g}^k(\yt{\gamma_i})$ intersects either $D_\alpha^{k/3}(\yt{\gamma_j})$ or $D_\alpha^{2k/3}(\yt{\gamma_j})$ for each $j$ (see Figure \ref{fig2}).  Let us verify that the same is true replacing $\yt{g}^k$ with $\yt{f^{3p}}=\yt{h}\yt{g}^k$.  Let $\yt{x}$ be the point where $\yt{\gamma_i}$ intersects $\yt{\alpha}$.  Starting at $\yt{x}$ and tracing the $C_+$ and $C_-$ sides of $\yt{\gamma_i}$, let $\yt{\sigma}_+$ and $\yt{\sigma}_-$ be the first lifts of components of $\del S' \cup \alpha$ encountered on each side respectively.  Note that, because $h$ induces the identity map on $S'$, the lift $\yt{h}$ preserves all the lifts of $\del S' \cup \alpha$ bordering $C_+$ and $C_-$, and therefore the half-spaces of $\yt{\sigma}_-$ and $D_\alpha^k(\yt{\sigma}_+)$ which do not contain $\alpha$ (and do contain the endpoints of $\yt{g}^k(\yt{\gamma_i})$).  Therefore these half-spaces contain the endpoints of $\yt{f^{3p}}(\yt{\gamma_i})$.  By their arrangement on the boundary at infinity, any geodesic between and not intersecting $D_\alpha^{k/3}(\yt{\gamma_i})$ and $D_\alpha^{2k/3}(\yt{\gamma_i})$ is forced to essentially intersect $\yt{f^{3p}}(\yt{\gamma_i})$; such geodesics include a translate of each $\gamma_j$ by either $D_\alpha^{k/3}$ or $D_\alpha^{2k/3}$.  Clearly, $\yt{f^{3p}}(\yt{\gamma_i})$ and $\yt{f^{3p}}(\yt{\gamma_j})$ are disjoint for $i\neq j$, and by equivariance, the $n^2$ points of intersection can be pulled back to $(\cup_i\gamma_i)\cap (C_+\cup C_- )$.  It is not hard to argue that any point on the quotient surface has at most two pre-image points on this union of segments.  Therefore at least $n^2/2$ unique intersections with $\gamma$ and $f^{3p}(\gamma)$ appear on the surface, as claimed.

\begin{figure}[H]
\begin{center}
\includegraphics[height=3in]{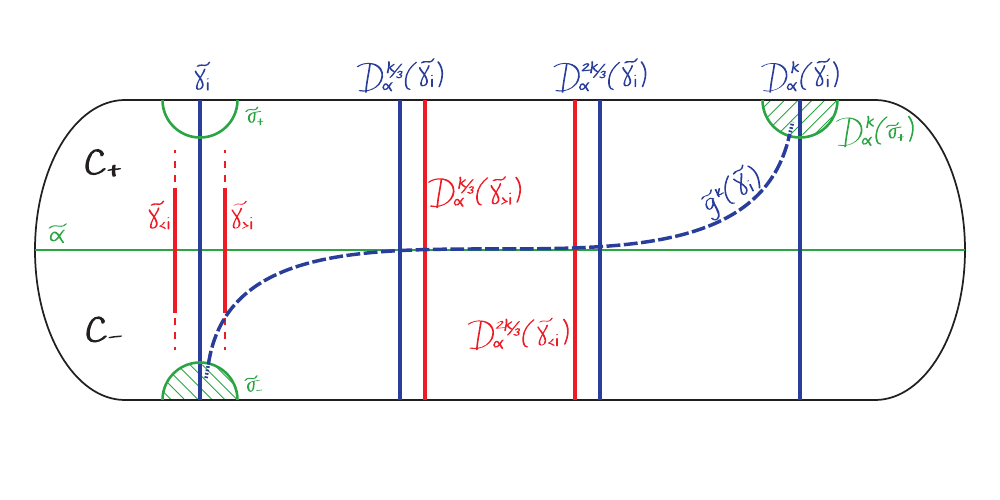}
\caption{Schematic of $\bH^2$ distorted to emphasize action by $D_\alpha$.  The endpoints of $\yt{g}^k(\yt{\gamma_i})$ lie within the shaded half-spaces, which are preserved by the map $\yt{h}$.  Note that $\yt{\gamma_j}$ for $j \neq i$ may or may not intersect $\yt{\sigma}_\pm$.}
\label{fig2}
\end{center}
\end{figure}

So far we have \[(M+1)^{3p\cdot\ell(f)+1} \geq M(M+1)^{\ell(f^{3p})} \geq i(f^{3p}(\gamma),\gamma) \geq i(\alpha,\gamma)^2/2,\]
assuming $f^p$ restricts on a subsurface $S'$ to Dehn twisting about $\alpha$; let us now bound the power $p$.  Let $S_1'$ be one of the at most two components of $S\setminus\sigma(f)$ whose boundary contains $\alpha$.  As before, for some $p_0 \leq |\chi(S)/\chi(S_1')|$, $f^{p_0}$ preserves $S_1'$.  Some power $p_1$ of $f^{p_0}$ induces an identity map on $S_1'$.  If $S_1'$ is the only component of $S\setminus\sigma(f)$ whose boundary contains $\alpha$ (so that $S'$ is $S_1'$ glued along $\alpha$), then $f^{p_0p_1}$ Dehn twists around a neighborhood of $\alpha$.  The other possibility is that $S'\setminus\alpha$ consists of two components $S_1'$ and $S_2'$; then we need an additional power $p_2$ so that $f^{p_0p_1p_2}$ induces identity maps in both $\Mod(S_1')$ and $\Mod(S_2')$, and can therefore only be Dehn twisting on $S'$.  For $i\in\{1,2\}$, the powers $p_i$ are bounded above by the maximal order of a mapping class of a surface $S_i'$ with boundary.  An upper bound is often given as $42|\chi(S_i')|$ based on the fact that the smallest hyperbolic 2-orbifold has area $\pi/21$; however, the smallest cusped hyperbolic 2-orbifold has area $\pi/6$ \cite{Siegel}.  For this application we can thus use an upper bound for $p_i$ of $12|\chi(S_i')|$.  We have \begin{equation}\label{eq:twist}(M+1)^{(3\cdot144\chi(S)^2\ell(f)+2)/2} \geq i(\alpha,\gamma).\end{equation}

Theorem \ref{t:computation} follows from comparing inequalities (\ref{eq:pA}) and (\ref{eq:twist}).\qed

\section{The list algorithm}

In order to find the canonical reducing system $\sigma(f)$ for a given mapping class $f$, we make use of Ivanov's elegant characterization, which generalizes to subgroups of $\Mod(S)$ \cite{Iv2}.  Under this expanded definition, $\sigma(f)$ is really a property of $\langle f \rangle$.  Recall that there is a power $N$, universal in $S$, such that for any $f \in \mathrm{Mod}(S)$, the mapping class $f^N$ is guaranteed to fix individually, rather than to permute, the components of $\sigma(f)$.  In particular, one may take $N = \xi(S)!$, where $\xi(S)$ is the maximal number of mutually disjoint, isotopically distinct curves on $S$, given by $3g + p - 3$ for a surface with $g$ genus and $p$ punctures.  Then the canonical reducing system $\sigma(f)$ is the set of simple closed curves $\alpha$ such that
\begin{itemize}
\item $f^N(\alpha)=\alpha$, and
\item if $i(\alpha,\beta) > 0$, then $f^N(\beta)\neq\beta$.
\end{itemize}

For our algorithm, we need to know we can find a ``test'' curve $\beta$ of short distance that will disqualify a candidate for $\sigma(f)$ that fulfills the first bullet but not the second.  We use the following:

\begin{lemma} \label{disqualifiers}For $f\in\Mod(S)$, write \[F=F(\ell_X(f))=C^{\ell_X(f)},\] where $X$ is the Artin--Dehn--Humphries--Lickorish generating set for $\Mod(S)$ and $C$ is the constant furnished by Theorem \ref{t:estimate}, so that every $\alpha \in \sigma(f)$ has $\ell(\alpha) \leq F$, where $\ell$ is the length function defined by geometric intersection number with $\Gamma$, the set of base curves for the generating twists.

Suppose that $\delta$ is a simple closed curve with $f^N(\delta)=\delta$ and $\ell(\delta) \leq F$.  If $\delta$ is not a component of $\sigma(f)$, then there exists $\beta$ such that $i(\delta,\beta)>0$ but $f^N(\beta)=\beta$, and $\ell(\beta) \leq 2F+4$.
\end{lemma}

\emph{Proof.} First observe that $\ell(\gamma) \leq 2$ for each $\gamma \in \Gamma$.  Since $\Gamma$ fills $S$, some $\gamma \in \Gamma$ must intersect $\delta$.  If this $\gamma$ is fixed by $f^N$, we may take $\beta = \gamma$.  So suppose $\gamma$ is not fixed by $f^N$.  It cannot be that $\gamma$ lies entirely in a pseudo-Anosov component of $f$, since then $\delta$ intersects that component, contradicting that $f^N$ fixes $\delta$.  Therefore $\gamma$ intersects some component of $\sigma(f)$.  Consider a point of intersection of $\gamma$ and $\delta$.  Let $a$ be the arc of $\gamma$ one obtains by tracing that curve in both directions until it first hits $\sigma(f)$ on each side.  This gives an essential arc of $\gamma \cap S'$ intersecting $\delta$ at least once, where $S'$ is the component of $S\setminus \sigma(f)$ to which $f$ restricts as the identity, and which contains $\delta$.  Let $\alpha$ consist of the one or two components of $\sigma(f)$ intersecting $a$ at its endpoints.  Since $a$ is an essential arc, at least one component of the regular neighborhood of $a \cup \alpha$ is an essential curve of $S$ lying completely in $S'$ and thus fixed by $f^N$; let this component be $\beta$.  By construction, the length of $\beta$ is less than the length of $\alpha$ plus twice the length of $\gamma$, which in turn is less than $2F+4$.  Also by construction, $\delta$ intersects $\beta$, and $f^N(\beta)=\beta$; thus $\beta$ is the desired disqualifying curve.\qed

\subsection{The list-and-check algorithm}\label{s:listalg}

We are given a word $f$ in Artin--Dehn--Humphries--Lickorish generators $X = \{T_{\gamma_i}\}$.  Let $F=F(\ell_X(f))$ be the upper bound on length of components of $\sigma(f)$ given by Theorem \ref{t:estimate} and referenced in Lemma \ref{disqualifiers}.  The algorithm proceeds as follows:

\begin{itemize}
\item[(1)] Populate the set \textsc{list} with the words representing all simple closed curves of length at most $2F+4$.
\item[(2)] For each curve $\gamma \in$ \textsc{list}, test if $f^N(\gamma)=\gamma$.  If this is true, add $\gamma$ to the set \textsc{test-list}.
\item[(3)] Let \textsc{shortlist} be the subset of curves in \textsc{test-list} with length at most $F$.
\item[(4)] For each curve $\gamma \in$ \textsc{shortlist}, test for intersection with curves in \textsc{test-list}.  If there are no intersections, then $\gamma \in \sigma(f)$.
\item[(5)] If $\sigma(f)$ is not empty, then $f$ is reducible and its canonical reducing system is given.  If $\sigma(f)$ is empty, but all the curves $\gamma_i$ appear in \textsc{shortlist}, then $f$ is finite-order.  Otherwise $f$ is pseudo-Anosov.
\end{itemize}

\subsection{Implementation and complexity}\label{s:complexity}
We need to encode simple closed curves so that we may compute their intersection numbers, and also their images under our fixed set of generators of $\mathrm{Mod}(S)$.  Here we describe an exponential-time implementation of our algorithm using a triangulation of the surface for the encoding, and computational methods described by Schaefer, Sedgwick, and Stefankovic \cite{SSS1, SSS2}.  Alternatively, our algorithm may be implemented using Dehn--Thurston coordinates as described in \cite{ThurstonD}.  However, both approaches are currently limited to mapping class groups of surfaces with punctures.  While either choice of coordinates certainly makes sense for curves on closed surfaces, we do not know how to compute intersection number in polynomial time.

A fixed triangulation of the surface allows one to describe a simple closed curve in terms of \emph{normal coordinates}: one takes a representative of the simple closed curve which intersects the edges of the triangulation a minimal number of times, and then one produces a list with entries for each edge of the triangulation, where the entries record how many times a representative of the curve intersects that edge.  We require that the representative avoids vertices and enters and exits each triangle through different edges, and we require that the vertices of the triangulation lie on the boundary of the surface; then two sets of normal coordinates represent the same curve if and only if the coordinates are equal \cite{SSS1}.  In this case, Schaefer-Sedgwick-Stefankovic have polynomial-time algorithms that compute the coordinates of a curve after a Dehn twist about another curve, and the intersection number between curves \cite{SSS2}.

Our list-and-check algorithm starts with a function $F(\ell_X(f))$ for identifying sufficiently short curves.  To be compatible with the algorithms of Schaefer-Sedgwick-Stefankovic, we use a new notion of length, which we call \emph{triangulation length}: the number of times a curve intersects the triangulation.  That is, where $T$ is a fixed minimal triangulation, for any curve $\gamma$, let $\ell_T(\gamma)$ be the sum of all entries in the (unique) normal coordinates for $\gamma$.  Now we must relate $\ell(\gamma)$ to $\ell_T(\gamma)$.  Recall that $X$ consists of twist and half-twist generators whose base curves fill the surface.  We may assume all the vertices of our triangulation lie on the boundary punctures, and that all edges of the triangulation are geodesics in some hyperbolic metric, so that minimal intersection is realized between any edge of the triangulation and any geodesic curve.  Furthermore, since all vertices are on the boundary, and minimal intersection allows no bigons, the count of intersections of any geodesic curve with each edge of the triangulation gives its normal coordinates.  The filling curves cut $S$ into a collection of disks and punctured disks.  Let $K$ be the maximum number of arcs of the triangulation through any of these components.  Observe that, for any curve $\gamma$, we have $K\cdot\ell(\gamma) > \ell_T(\gamma)$.  Therefore the task of listing all curves with $\ell < L$ is fulfilled by listing all curves with $\ell_T < LK$.  We can create \textsc{list} in Step (1) of the algorithm using this encoding, in time polynomial in $F$, and thus exponential in $\ell_X(f)$.  We have the necessary algorithms for Steps (2) and (4) by Theorem 2 of \cite{SSS2}, again in time polynomial in $F$.  The complexity of the entire list-and-check algorithm implemented this way is exponential in word length.

We know of a third approach to curve-related computations, by encoding curves as words in the fundamental group of the surface.  Then it is straightforward to compute images of a curve under products of a finite list of generators, and \cite{BirmanSeries84, BirmanSeries87, CohenLustig} give polynomial-time algorithms to count geometric intersection number, whether or not the surface has boundary.  However, we do not know how to enumerate only words that represent simple curves, short of enumerating all words and checking for self-intersection.  This way, the list-and-check algorithm requires time doubly exponential in word length.

\section{The cover algorithm}
In this section, we prove Theorem \ref{t:covers}.  Let $X\subset\Mod(S)$ be a finite generating set for $\Mod(S)$.  Fixing a basepoint on $S$, we obtain $\Mod(S)$ as a quotient of $\Mod^1(S)$, which is the group of based homotopy classes of orientation-preserving homeomorphisms of $S$.  Viewing the chosen basepoint as a basepoint for $\pi_1(S)$, we obtain an action of $\Mod^1(S)$ on $\pi_1(S)$.  We fix a finite generating set $Y$ for $\pi_1(S)$.  Given $Y$, we compute conjugacy representatives $\{p_1,\ldots,p_n\}\subset\pi_1(S)$ of small loops about each puncture of $S$.  In all of our calculations, we disregard the conjugacy classes of these elements and their powers.  This is a reasonable assumption, since once we identify conjugacy representatives of small loops about the punctures of $S$, the solution to the conjugacy problem in $\pi_1(S)$ gives us an effective way of determining if a homotopy class of loops on $S$ is boundary parallel (see Section \ref{s:effective}).

In what follows, $\ell_X(\cdot)$ denotes word length in $\Mod(S)$ with respect to $X$, and likewise $\ell_Y(\cdot)$ is length in $\pi_1(S)$ using $Y$.
We can lift $X$ to $\yt{X}\subset\Mod^1(S)$ by choosing arbitrary preimages of the mapping classes in $X$.  Such a choice gives us a set-theoretic splitting $\Mod(S)\to\Mod^1(S)$.  For $x\in X$, the splitting $x\mapsto \yt{x}$ allows us to determine a (maximum) \emph{expansion factor} $\lambda$ for mapping classes in $\Mod(S)$, which we define by
\[\lambda=\max_{x\in X}\max_{y\in Y}\ell_Y(\yt{x}(y)).\]
It is immediate from the definition of the expansion factor that if $f\in\Mod(S)$ satisfies $\ell_X(f)=n$ and $w\in\pi_1(S)$, we have
\begin{equation}\label{e:expansionfactor}
\ell_Y(\yt{f}(w))\leq\lambda^n\cdot\ell_Y(w)
\end{equation}
where $\yt{f}$ is any lift of $f$ in $\grp{\yt{X}}$ of length $n$.  Thus, if $[w]$ is the conjugacy class of $w$ in $\pi_1(S)$, then the shortest representative for $f([w])$ in $\pi_1(S)$ has length at most $\lambda^n\cdot \ell_Y(w)$.

Given a mapping class $f\in\Mod(S)$ presented as a word in the generating set $X$, we would like to find a finite cover $\yt{S}\to S$ to which $f$ lifts, whose degree is controlled by a computable function in $\ell_X(f)$, and such that the action of any lift of $f$ on $H_1(\yt{S},\bC)$ determines the Nielsen--Thurston classification (compare with Theorem \ref{t:covers}).  The idea is to build $\yt{S}$ as a regular cover of $S$ and to exploit the interplay between the action of the deck group and the action of a lift of $f$ on $H_1(\yt{S},\bC)$.  The proof of Theorem \ref{t:covers} is fairly quick after we gather the necessary facts.

\begin{lemma}[Lemma 3.1 in \cite{KobGeom}]\label{l:faithful}
Let $\yt{S}\to S$ be a regular finite cover with deck group $Q$.  Then the action of $Q$ on $H_1(\yt{S},\bC)$ is faithful.  Moreover, we have an isomorphism of $Q$-representations: \[H_1(\yt{S},\bC)\cong (\bC[Q]/\bC)^{|\chi(S)|}\oplus \bC^{b_1(S)}.\]  Here, $\bC[Q]/\bC$ denotes the regular representation of $Q$ modulo its trivial summand.
\end{lemma}


In what follows, we identify each essential, nonperipheral closed curve $c$ on $S$ with the pair of nontrivial conjugacy classes $\{[c],[c^{-1}]\} \subset \pi_1(S)$ it determines.  It is well-known that a mapping class $f \in \Mod(S)$ fixes the isotopy class of $c$ if and only if it preserves $\{[c],[c^{-1}]\}$ set-wise.

\begin{lemma}\label{l:conjsep}

Given $\{c_1,\dots,c_s\}$ a finite list of essential, nonperipheral homotopy classes in $\pi_1(S)$, and $f \in \Mod(S)$, there exists a finite characteristic quotient $Q$ of $\pi_1(S)$ such that if $f(c_i)\neq c_i$ as isotopy classes of curves, then the images in $Q$ of the pairs $\{[c_i],[c_i^{-1}]\}$ and $\{f([c_i]),f([c_i^{-1}])\}$ belong to distinct (conjugacy class, inverse conjugacy class) pairs in $Q$.  Furthermore, $|Q| < N$, where $N$ may be computed from $\ell_X(f)$, the length $s$ of the list of curves, and an upper bound for $\ell_Y(c_i)$.
\end{lemma}

\begin{proof} When $s=1$, the lemma follows from the effective conjugacy separability in free groups and surface groups (see Section \ref{s:effective}), and the expansion factor observation noted in (\ref{e:expansionfactor}) above.  In general, one can produce a quotient $Q_i$ for each conjugacy class and then intersect the kernels.  The intersection of the kernels will again be a characteristic subgroup of $\pi_1(S)$, and the index of the intersection will be controlled by the product of the orders of the quotients. \end{proof}

\begin{lemma}\label{l:characters}  Let $\yt{S}$ be the connected cover of $S$ determined by a finite characteristic quotient $Q$ of $\pi_1(S)$.  Suppose that $g \in \pi_1(S)$ and its image under $f \in \Mod(S)$ have non-conjugate images in $Q$.  Then any lift $\yt{f}$ of $f$ acts nontrivially on $H_1(\yt{S},\bC)$; in particular, the images of $g$ and $f(g)$ in $Q$ induce different characters on $H_1(\yt{S},\bC)$ viewed as a representation of $Q$.
\end{lemma}

\begin{proof}
Since $Q$ is a characteristic quotient of $\pi_1(S)$, the mapping class $f$ lifts to $\yt{S}$, and we choose a fixed lift $\yt{f}$.  The induced isomorphism of vector spaces $\yt{f}_*:H_1(\yt{S},\bC)\to H_1(\yt{S},\bC)$ generally depends on the choice of lift, but we claim that its nontriviality is independent of the choice of lift.  To see this, we fix a representative of $f\in\Aut(\pi_1(S))$, which we will also call $f$.  This way, we get an unambiguous action of $f$ on $H_1(\yt{S},\bC)$ as well as on the quotient group $Q$.  We will show that no matter what representative of $f$ we choose in $\Aut(\pi_1(S))$, the action of $f$ on $H_1(\yt{S},\bC)$ is nontrivial.

Note that the naturality of the $f$-- and $Q$--actions on $H_1(\yt{S},\bC)$ imply that $q\cdot f(v)=f(f^{-1}(q)\cdot v)$, for any $q \in Q$ and $v \in H_1(\yt{S},\bC)$.  If $q$ is the image of $g$ in $Q$, we have that $f(q)$ is the image of $f(g)$ in $Q$.  The hypothesis that $q$ and $f(q)$ are not conjugate implies that there is an irreducible character $\chi$ of $Q$ such that $\chi(q)\neq\chi(f^{-1}(q))$; this stems from the fact that the irreducible complex characters of $Q$ span the ring of class functions on $Q$.  Note in particular that the computation $\chi(q)\neq\chi(f^{-1}(q))$ is independent of the choice of representative of $f$ in $\Aut(\pi_1(S))$, because characters are invariant under conjugation.  Here, we precompose the character $\chi$ with $f^{-1}$ instead of $f$ in order to get a left action.  In the decomposition of $H_1(\yt{S},\bC)$ into a direct sum of irreducible representations of $Q$, the summands corresponding to $\chi$ are mapped by $f$ to the summands corresponding to $\chi\circ f^{-1}$, and these are distinct because their corresponding characters evaluate $q$ differently.
\end{proof}

\begin{lemma}\label{l:shortfill}
There exists an $L$ such that the (finite) set $F$ of elements of $\pi_1(S)$ of length at most $L$ represents a collection of closed curves on $S$ which fills $S$.
\end{lemma}

\begin{proof}  This is obvious after fixing any finite, filling set of curves on $S$.\end{proof}

\begin{proof}[Proof of Theorem \ref{t:covers}]
Let $f$ be given of length $n$.  We first check if $f$ is nontrivial and whether or not it has finite order.  There is a $k$ which is independent of $f$ such that $f^k$ is trivial in $\Mod(S)$ if and only if $f$ has finite order.  We have that $f^k$ is trivial if and only if $f^k$ preserves the conjugacy classes of each element in $F$ from Lemma \ref{l:shortfill}.  We lift $f$ to $\yt{f}$ in $\Mod^1(S)$ and apply both $\yt{f}$ and $\yt{f}^k$ to each element of $F$.  This process gives us a finite set of words of $\pi_1(S)$ with lengths at most $\lambda^{nk}\cdot L$, where $\lambda$ is the expansion factor determined by our generating sets for $\Mod(S)$ and $\pi_1(S)$.

By Lemma \ref{l:conjsep}, there is a finite characteristic quotient $Q$ of $\pi_1(S)$ such that any two non-conjugate elements of $\pi_1(S)$ of length at most $\lambda^{nk}\cdot L$ have non-conjugate images in $Q$.  Fix $g\in\pi_1(S)$ such that $g$ and $f(g)$  both have length at most $\lambda^{nk}\cdot L$, and write $q$ and $f(q)$ for the images of $g$ and $f(g)$ respectively in $Q$.  Let $\yt{S}$ be the connected cover of $S$ determined by the quotient $Q$; the degree of $\yt{S}$ is bounded above by the $N$ from Lemma \ref{l:conjsep}.  

By Lemma \ref{l:characters}, we may check the values of the various characters of $Q$ on the images of the elements of $F$, and determine whether or not $f$ is the identity mapping class.  In particular, we can determine whether or not $f$ is the identity mapping class by examining its action on $H_1(\yt{S},\bC)$.  A similar observation concerning the triviality of $f^k$ shows that $\yt{S}$ is a classifying cover for finite-order mapping classes on $S$.

When $f^k$ is nontrivial, it remains to determine whether $f^k$ (and hence $f$) is reducible or pseudo-Anosov.  We may assume $k$ is such that $f^k$ is reducible if and only if it fixes a curve.  Using Theorem \ref{t:estimate}, we enumerate a list of pairs of conjugacy classes
\[\{[g_1^{\pm1}],\ldots,[g_s^{\pm1}]\}\]
such that if $f^k$ is reducible then $f^k$ stabilizes at least one of these pairs of conjugacy classes. 
If $\lambda$ is the expansion factor for our generating set for $\Mod(S)$ then there are representatives for the conjugacy classes \[\{f^k([g_1^{\pm1}]),\ldots,f^k([g_s^{\pm1}])\}\] whose lengths are at most $(C\lambda)^{nk}$.

The rest of the proof is identical to the finite order case.  By Lemma \ref{l:conjsep}, there is a finite characteristic quotient $Q$ of $\pi_1(S)$ such that any two non-conjugate elements of $\pi_1(S)$ of length at most $(C\lambda)^{nk}$ have non-conjugate images in $Q$.  As before, we take $\yt{S}$ to be the connected cover of $S$ determined by the quotient $Q$.  The mapping class $f$ is pseudo-Anosov if and only if the characters induced on $H_1(\yt{S},\bC)$ by the images of $[g]$ and $f([g])$ in $Q$ are not equal for each $g\in\pi_1(S)$ of length at most $C^{nk}$.
\end{proof}

\section{Effective conjugacy separability}\label{s:effective}
In this section, we will give a geometric proof that free groups and surface groups are effectively conjugacy separable.  That is to say, given any non-conjugate pair of elements $g,h\in G$ where $G$ is finitely generated by a set $\mathcal{S}$, there is a finite quotient $Q_{g,h}$ of $G$ in which $g$ and $h$ are not conjugate, and such that the order of $Q_{g,h}$ is controlled by a computable function of $\ell_\mathcal{S}(g)+\ell_\mathcal{S}(h)$, where $\ell_\mathcal{S}(\cdot)$ denotes word length with respect to $\mathcal{S}$.  We will handle the free group and closed surface group cases separately, though the arguments are essentially the same.

\subsection{From quotients to characteristic quotients}\label{ss:char}
We first gather some general facts about $p$--groups and characteristic quotients.  The first is a trivial observation.

\begin{lemma}
Let $g,h\in G$, let $G\to Q$ be a quotient of $G$ in which the images of $g$ and $h$ are not conjugate, and let $G\to\overline{G}$ be a quotient of $G$ which factors the quotient map $G\to Q$.  Then the images of $g$ and $h$ are not conjugate in $\overline{G}$.
\end{lemma}

Let $p$ be a fixed prime.

\begin{lemma}
Let $G$ be a group and let $H,K$ be a normal subgroups of $G$ of $p$--power index.  Then $H\cap K$ has $p$--power index in $G$.
\end{lemma}
\begin{proof}
This is an easy consequence of the Second Isomorphism Theorem for groups.
\end{proof}

We remark that in the previous lemma, the normality of $H$ and $K$ is essential.

\begin{cor}
Let $G$ be an $m$--generated group and let $H<G$ be a normal subgroup of index $q=p^n$.  Then there exists a subgroup $K<H<G$ such that $K$ is characteristic in $G$ and has index at most $q^{q^m}$.
\end{cor}
\begin{proof}
Write $Q=G/H$, which is a finite group of order $q$.  There are at most $q^m$ homomorphisms from $G$ to $Q$, and the intersection $K$ of the kernels of these homomorphisms has index at most $q^{q^m}$.
\end{proof}

\subsection{Free groups}
The argument we give here is essentially a geometric version of the argument given in \cite{LS}. 
\begin{prop}\label{freeconjsep}
Finitely generated free groups are effectively conjugacy separable.
\end{prop}
\begin{proof}
Let $F$ be a free group on a finite set $\mathcal{S}$.  We may assume $\mathcal{S}$ has at least two elements, since it is easy to see that, when $F$ is cyclic, we may use $Q_{g,h}$ of order \[p^{\max\{\ell_\mathcal{S}(g), \ell_\mathcal{S}(h)\}}.\]  Let $X$ be a wedge of circles with vertex $x$, whose edges are labelled by the elements of $\mathcal{S}$.  We choose orientations on the edges arbitrarily and choose a midpoint of each edge of $X$.  If $\gamma$ is an oriented based loop in $X$, we can record the sequence of midpoints traversed by $\gamma$ while keeping track of orientations.  This gives us an identification of $\pi_1(X,x)$ with $F$.  If $\gamma$ does not backtrack, its  \emph{combinatorial length} is defined to be the absolute number of midpoints traversed by $\gamma$.  The combinatorial length of $\gamma$ coincides with $\ell_\mathcal{S}(g)$ for $g$ the element of $F$ represented by $\gamma$.

We will use induction on the combinatorial length $N$ to take two elements $w,v\in F$ with $\ell_\mathcal{S}(w)+\ell_\mathcal{S}(v)=N$, and produce a $p$--group quotient of $F$ which has order controlled by a computable function of $N$, and in which the images of $w$ and $v$ are conjugate if and only if they are conjugate in $F$.

So, fix a prime $p>2$, and let $w,v\in F$ be reduced and cyclically reduced words, represented by non-backtracking unbased loops in $X$ with $\ell_\mathcal{S}(w)+\ell_\mathcal{S}(v)=N$.  Throughout the proof, we allow for the case that one of $v$ or $w$ is trivial, so that the corresponding loop is just a point.

If $N\leq 2$ then either only one of $v$ and $w$ is nontrivial, or $\ell_\mathcal{S}(v)=\ell_\mathcal{S}(w)=1$, in which case $v$ and $w$ are conjugate if and only if they are equal, which happens if and only if their images in $H_1(F,\bZ/p\bZ)$ are equal (since $p>2$).  If only one of $v$ and $w$ is nontrivial, say $v$, then the image of $v$ in $H_1(F,\bZ/p\bZ)$ is nontrivial whereas the image of $w$ is trivial.  Indeed, $v$ is a reduced word in $F$ of combinatorial length exactly two and is therefore a product of at most two generators of $F$ which are not inverses of each other.  Again, since $p\neq 2$, we have that the image of $v$ in $H_1(F,\bZ/p\bZ)$ must be nontrivial.

If loops representing $v$ and $w$ are concentrated in a single wedge summand of $X$, so that these loops only intersect one midpoint of one edge of $X$, then $v$ and $w$ generate a cyclic subgroup of $F$.  In this case, we may look at the images of $v$ and $w$ in $H_1(F,\bZ/p^N\bZ)$.  It is clear that $v$ and $w$ will have distinct images in this quotient if and only if $v\neq w$.

For the general inductive step, we can make several simplifying assumptions.  The first is that the span of the images of $v$ and $w$ in $H_1(F,\bZ/p\bZ)$ has dimension at most one, after replacing $v$ or $w$ by an appropriate conjugate.  Indeed, otherwise $v$ and $w$ are not conjugate since their homology classes are linearly independent after applying any conjugacy.  Secondly, we may assume that $v$ and $w$ \emph{fill} $X$, in the sense that they are not simultaneously homotopic into a proper subcomplex of $X$, because in that case we may restrict our attention to the proper subgraph of $X$ filled by $v$ and $w$.

We have that the loops representing $v$ and $w$ fill all of $X$.  Note that the subgroup generated by $v$ and $w$ may still be cyclic even though all of $X$ is filled.  By the first assumption, there is a homomorphism $\phi:\pi_1(X)\to\bZ/p\bZ$ for which $v$ and $w$ are both in the kernel.  Thus, loops representing $v$ and $w$ lift to the cover $Y\to X$ determined by $\phi$.  We pull back that edge-midpoint structure on $X$ to $Y$.  Choose lifts of the two loops representing $v$ and $w$, which we call $v_Y$ and $w_Y$ respectively.  Different choices of $v_Y$ and $w_Y$ differ by the action of the deck group, so choosing a particular lift does not affect conjugacy considerations.  Therefore, we choose $v_Y$ and $w_Y$ so that they span a one--dimensional subspace of $H_1(Y,\bZ/p\bZ)$, if possible.

Observe that the combinatorial lengths of $v_Y$ and $w_Y$ are equal to those of $v$ and $w$ respectively, with respect to the pulled--back edge-midpoint structure on $Y$.  By the second assumption, namely the assumption that $v$ and $w$ fill $X$, there is an edge of $Y$ which is not a closed loop and which is traversed by at least one of $v$ and $w$.  This edge can be extended to a maximal tree $T$ in $Y$.  A free basis for $\pi_1(Y)$ is identified with $Y\setminus T$.

Now consider the images $\overline{v_Y}$ and $\overline{w_Y}$ of $v_Y$ and $w_Y$ in the wedge of circles $Y/T$.  The midpoints of edges in $Y\setminus T$ give rise to new combinatorial lengths for $\overline{v_Y}$ and $\overline{w_Y}$ as loops in $Y/T$.  By the choice of $T$, we have that the sum of the combinatorial lengths of $\overline{v_Y}$ and $\overline{w_Y}$ is at most $N-1$.  The conclusion follows by induction on $N$.
\end{proof}

We remark that the exact rank of $F$ or of $\pi_1(Y/T)$ in the proof above is irrelevant for the induction.  Furthermore, we remark that if the conjugacy classes of elements $v$ and $w$ of $F$ are separated in a $p$--group quotient $Q$ of $F$ with kernel $K$, then by Subsection \ref{ss:char} we may produce a characteristic $p$--group quotient $\yt{Q}$ of $F$ which admits $Q$ as a further quotient, and whose size is controlled by the order of $Q$.

\subsection{Closed surface groups}
We will restrict our attention to orientable hyperbolic-type surfaces.
\begin{prop}\label{surfaceconjsep}
Let $S$ be a closed surface.  Then $\pi_1(S)$ is effectively conjugacy separable.
\end{prop}

Conjugacy separability of surface groups has been studied by various authors.  H. Wilton~\cite{wilton} gives an effective proof which relies on the work of Niblo~\cite{niblo}.

\begin{proof}[Proof of Proposition \ref{surfaceconjsep}]
Let $S$ be a closed surface of genus at least two.  We fix a hyperbolic structure on $S$ such that a fundamental domain $F$ for the action of $\pi_1(S)$ on $\bH^2$ is given by a convex right-angled polygon.  By taking finite connected unions of $F$ along its faces, we can build fundamental domains for finite (possibly non-regular) covers of $S$.

Write $p:\bH^2\to S$ for the universal covering map.  We assume that the boundary of the polygon $F$ descends to a set of non-separating geodesic curves on $S$ with linearly independent homology classes.  Call this collection of non-separating curves $\mathcal{X}$.  We tile $\bH^2$ by $\pi_1(S)$--translates of $F$, whose edges fit together into the collection $\mW=p^{-1}(\mathcal{X})$ of bi-infinite geodesics of $\bH^2$.  We orient the curves of $\mathcal{X}$ arbitrarily and pull back these orientations to $\mW$.

Let $\gamma$ be an oriented geodesic representative of a primitive free homotopy class of closed curves in $\pi_1(S)$, and write $\ell_{hyp}(\gamma)$ for its hyperbolic length.  The preimage $\yt{\gamma}=p^{-1}(\gamma)$ in $\bH^2$ is a union of geodesics, which generally will not be pairwise disjoint.  Assume for now that $\gamma$ is not a curve of $\mathcal{X}$.  Choose a point $x$ in the interior of $\yt{\gamma}$ which is disjoint from $\mW$ and which sits on exactly one bi-infinite geodesic $\delta\subset\yt{\gamma}$.  Following $\delta$ in the positive direction for a distance of $\ell_{hyp}(\gamma)$ and counting (without multiplicity) the number of intersection points with $\mW$ gives a well-defined notion of length which we call the \emph{combinatorial length} of $\gamma$, and denote by $\ell_c(\gamma)$.  It is easy to see that this number is independent of the choice of $x$.  The combinatorial length of $\gamma$ is also the number of interiors of distinct fundamental domains (from our tiling by translates of $F$) which a subsegment of $\yt{\gamma}$ of length $\ell_{hyp}(\gamma)$ intersects, minus one.  Observe that the combinatorial length of $\gamma$ can be computed on $S$ be counting the total number of intersections with $\mathcal{X}$.  In the case $\gamma \in \mathcal{X}$, we make the convention that $\ell_c(\gamma) = 1$.

If $g=h^n\in\pi_1(S)$ is the $n^{th}$ power of a primitive free homotopy class then we think of the corresponding geodesic $\gamma_g$ representing $g$ to be the same as the geodesic $\gamma_h$ representing $h$, except that $g$ traverses $\gamma_h$ a total of $n$ times.  We call the geodesic $\gamma_g$ the \emph{underlying geodesic} for $g$, and we define the combinatorial length of $g$ by $\ell_c(g)=n\cdot\ell_c(\gamma_h)$.  Because $\ell_c(g)$ can be understood as the minimal number of times a curve representing $g$ intersects $\mathcal{X}$, we have $\ell_c(g) \leq K\ell_\mathcal{S}(g)$, where $\mathcal{S}$ is a fixed generating set for $\pi_1(S)$ and $K=\max_{s\in\mathcal{S}}\{\ell_c(s)\}$.
By construction, a nontrivial free homotopy class in $\pi_1(S)$ has nonzero combinatorial length.

For $S'\to S$ a finite cover, we may choose a fundamental domain $F'$ for $S'$ such that $F'$ is a connected union of copies of $F$, and translates of $F'$ tile $\bH^2$.  Specifically, consider the cell decomposition of $S'$ given by taking the pre-image of $\mathcal{X}$.  We may delete edges of adjoining faces until the cell decomposition has a single $2$-cell, whose pre-images in $\bH^2$ we use to define the boundaries of the copies of $F'$.  Let $\mathcal{X}'$ be the $1$-skeleton on $S'$ of this new cell decomposition.  If $\gamma \notin \mathcal{X}$ is a closed geodesic on $S$ which lifts to $\yt{\gamma}$ on $S'$, we define the combinatorial length of $\yt{\gamma}$ to be the number of tiles of $F'$ in $\bH^2$ traversed by a subsegment of $\yt{\gamma}$ of length $\ell_{hyp}(\gamma)$ which starts in the interior of a tile, minus one.  We make the convention that, if $\gamma \in \mathcal{X}$ lifts to $S'$, the combinatorial length of its lift is again one.  By construction, the combinatorial length of a geodesic on $S'$ is bounded above by the combinatorial length of its quotient on $S$.  It is also precisely the number of points of transverse intersection of the geodesic representative $\gamma$ with $\mathcal{X}'$, a piecewise-geodesic graph on $S'$.  If \[S^{i} \to S^{i-1} \to \cdots \to S\] is a tower of covers, we may construct the corresponding fundamental domains compatibly so that $F^i$ is a concatenation of copies of $F^{i-1}$, and the preimage of $\mathcal{X}^i$ in $\bH^2$ is a proper subset of the preimage of $\mathcal{X}^{i-1}$.

As before, for $g=h^n\in\pi_1(S')\subset\pi_1(S)$ the $n^{th}$ power of a primitive free homotopy class, we make a new definition of combinatorial length in $\pi_1(S')$ by considering the underlying geodesic $\yt{\gamma}_h$ in $S'$, and setting $\ell_{c'}(g)=n\cdot\ell_{c'}(\gamma_h)$.

As in the case of free groups, we proceed by induction on the sum of the combinatorial lengths of two free homotopy classes in $\pi_1(S)$.  There are two different base cases to consider: the case where both $g_1$ and $g_2$ are nontrivial, and the case where only $g_1$ is nontrivial.  Let $\gamma_1$ and $\gamma_2$ be closed, oriented geodesics on $S$ underlying two (possibly trivial) free homotopy classes $g_1$ and $g_2$ in $\pi_1(S)$, and let $p>2$ be a fixed prime.  Suppose \[\ell_c(g_1)+\ell_c(g_2)\leq 1.\]  In this case, we have that only $g_1$ is nontrivial.  We see that $\gamma_1$ intersects exactly one curve in $\mathcal{X}$ exactly once and therefore represents a primitive homology class of $S$ which survives in $H_1(S,\bZ/p\bZ)$.  For the second base case, we have $g_1$ and $g_2$ are nontrivial and therefore \[\ell_c(g_1)+\ell_c(g_2)=2.\]  Since both $g_1$ and $g_2$ are nontrivial, we have $\ell_c(g_1)=\ell_c(g_2)=1$ and hence $\gamma_1$ and $\gamma_2$ each intersect some curve in $\mathcal{X}$ exactly once.  We thus obtain that $g_1$ and $g_2$ are either equal and therefore conjugate, or they are distinct in $H_1(S,\bZ/p\bZ)$ and therefore not conjugate in $\pi_1(S)$ (here, we are using the fact that $p\neq 2$).

Now suppose that $\ell_c(g_1)+\ell_c(g_2)=N>2$.  Let us first consider the case that $\gamma_1$ and $\gamma_2$ fill an annulus $A$ of $S$.  If $A$ is nonseparating then we look at the images of $g_1$ and $g_2$ in $H_1(S,\bZ/p^N\bZ)$ and note that they will be equal if and only if $g_1=g_2$ as free homotopy classes.  If $A$ is separating, then we consider the cover $S_p\to S$ classified by the quotient $H_1(S,\bZ/p\bZ)$.  Clearly $A$ lifts to $S_p$, and each lift of $A$ to $S_p$ is nonseparating.  Choose lifts of $g_1$ and $g_2$ whose underlying geodesics lie within a fixed lift of $A$ to $S_p$.  Again, we have that the the images of these lifts in $H_1(S_p,\bZ/p^N\bZ)$ will be equal if and only if $g_1=g_2$ as free homotopy classes.

Now assume that $\gamma_1$ and $\gamma_2$ fill a subsurface $S_0\subset S$ which is not an annulus.  In particular, at least two distinct elements of $\mathcal{X}$ intersect $S_0$ essentially.  As in the free group case, we may assume that the span of the homology classes of $\gamma_1$ and $\gamma_2$ in $H_1(S,\bZ)$ has rank at most one, by replacing $g_1$ by an appropriate conjugate in $\pi_1(S)$.    Finally, we may assume that under the inclusion $S_0\subset S$, the image of $H_1(S_0,\bZ)\to H_1(S,\bZ)$ has rank at least two.

To see this last claim, if $S_0$ has positive genus then there is nothing to do.  Otherwise, we may assume that $S_0$ has genus zero and at least three boundary components.  We describe the case where $S_0$ has exactly three boundary components, with the general case being a straightforward generalization.  If $D$ is a boundary component of $S_0$ which is separating in $S$, then it is easy to see that there is a degree $p$ cover of $S$ to which $S_0$ lifts and wherein every lift of $D$ becomes nonseparating.  Thus, we can replace $S$ by an abelian $p$--cover of $S$ to which $S_0$ lifts and in which each boundary component of each lift of $S_0$ is nonseparating.  In this case, we have that the rank of the map $H_1(S_0,\bZ)\to H_1(S,\bZ)$ is exactly two.  Indeed, suppose the contrary.  We have that each boundary component of $S_0$ represents a primitive homology class in $H_1(S,\bZ)$.  Choosing a pair of boundary components $b_1,b_2\subset S_0$, we have that $b_1$ and $b_2$ together separate $S$.  But then the third boundary component $b_3\subset S_0$ must separate $S$ as well, which is a contradiction since the homology class of $b_3$ in $S$ is nonzero.  This establishes the claim.

We choose a homomorphism $\phi:\pi_1(S)\to\bZ/p\bZ$ such that the homology classes of $\gamma_1$ and $\gamma_2$ lie in the kernel of $\phi$, so that the free homotopy classes $g_1$ and $g_2$ also lie in the kernel of $\phi$, and we choose a component of the preimage of $S_0$ in the corresponding cover.  Since $S_0$ is not an annulus and since $\gamma_1$ and $\gamma_2$ span a one--dimensional subspace of $H_1(S,\bZ)$, we can choose $\phi$ so that the homology class of some homologically nontrivial simple closed curve $\alpha\subset S_0$ does not lie in the kernel of $\phi$.  Let $S_{\phi}$ be the cover of $S$ determined by the homomorphism $\phi$.  A fundamental domain $F_{\phi}$ for $S_{\phi}$ in $\bH^2$ can be obtained by concatenating $p$ copies of the fundamental domain $F$ for $S$ together, and the $\pi_1(S_\phi)$--translates of $\partial F_{\phi}$ are a proper subset of $\mathcal{W}$, which descends to a $1$--skeleton $\mathcal{X}_{\phi}$ on $S_{\phi}$.

Choose any lifts $\yt{\gamma_1}$ and $\yt{\gamma_2}$ of $\gamma_1$ and $\gamma_2$.  Note that if the total combinatorial length \[\ell_{c_{\phi}}(\yt{\gamma_1})+\ell_{c_{\phi}}(\yt{\gamma_2})\] on $S_{\phi}$ is not strictly smaller than \[\ell_c(\gamma_1)+\ell_c(\gamma_2)\] on $S$, then the subsurface $S_0$ filled by $\gamma_1$ and $\gamma_2$ lifts to $S_{\phi}$ as well.  This can be seen by considering a regular neighborhood of the segment of the lift of $\gamma_i$ to $\bH^2$ used to compute combinatorial length, for each $i$; each time $\gamma_i$ crosses to a new tile of $F$, it also crosses to a new tile of $F_\phi$.  However, we produced a curve $\al\subset S_0$ above which does not lift, which is a contradiction.  The induction follows.
\end{proof}

\bibliography{biblio}
\thispagestyle{empty}
\end{document}